\UseRawInputEncoding

\documentclass[11pt,a4paper,reqno]{amsart}

\usepackage{setspace}
\usepackage{fullpage}
\usepackage{datetime}
\usepackage{xcolor}

\usepackage{enumitem}
\newcommand\rmlabel{\upshape({\itshape\roman*\,\/})}

\usepackage{amsmath}%
\usepackage{amsfonts}%
\usepackage{amssymb}%
\usepackage{graphicx}
\usepackage{mathrsfs}
\usepackage{hyperref}
\usepackage[margin=1in]{geometry}
\usepackage[normalem]{ulem}

\usepackage{cite}

\newtheorem{theorem}{Theorem}
\theoremstyle{plain}

\newtheorem{definition}[theorem]{Definition}
\newtheorem{example}{Example}

\newtheorem{lemma}[theorem]{Lemma}

\newtheorem{proposition}[theorem]{Proposition}

\newtheorem{question}[theorem]{Question}

\numberwithin{equation}{section}
\numberwithin{theorem}{section}
\numberwithin{case}{section}

\newtheorem*{lemma*}{Lemma~\ref{lem:abs}}
\newtheorem*{Lemma*}{Lemma~\ref{lem:con}}
\newtheorem*{LemmA*}{Lemma~\ref{lem:cov}}
\newtheorem*{lemmA*}{Lemma~\ref{lem:res}}

\usepackage{pgfplots} 
\usepackage{tikz}
\usetikzlibrary{calc}
\usetikzlibrary{arrows,shapes,chains}
\pgfdeclarelayer{background}
\pgfdeclarelayer{foreground}
\pgfsetlayers{background,main,foreground}
\usetikzlibrary{decorations.pathreplacing}

\def\COMMENT#1{}
\let\COMMENT=\footnote
\begin{document}

\title{On powers of tight Hamilton cycles in randomly perturbed hypergraphs}

\author[Y.~Chang]{Yulin Chang}
\address{(Y. Chang) Data Science Institute, Shandong University, Jinan, 250100, China}
\email{ylchang93@163.com}

\author[J.~Han]{Jie Han}
\address{(J. Han) School of Mathematics and Statistics, Beijing Institute of Technology, Beijing, 100000, China}
\email{jie\_han@uri.edu}

\author[L.~Thoma]{Lubos Thoma}
\address{(L. Thoma) Department of Mathematics, University of Rhode Island, Kingston, RI, 02881, USA}
\email{thoma@uri.edu}

\thanks{The research of the first author was supported by the China Postdoctoral Science Foundation (2022M711926, 2022T150386), Natural Science Foundation of China (12201352) and
Natural Science Foundation of Shandong Province (ZR2022QA083). The research of the second author was supported by Simons Foundation \#630884}

\begin{abstract}
For integers $k \geq 3$ and $r\geq 2$, we show that for every $\alpha> 0$, there exists $\varepsilon > 0$ such that the union of $k$-uniform hypergraph on $n$ vertices with minimum codegree at least $\alpha n$ and a binomial random $k$-uniform hypergraph $G^{(k)}(n,p)$ with $p\geq n^{-{\binom{k+r-2}{k-1}}^{-1}-\varepsilon}$ on the same vertex set contains the $r^{th}$ power of a tight Hamilton cycle with high probability.
Moreover, a construction shows that one cannot take $\varepsilon > C\alpha$, where $C=C(k,r)$ is a constant.
Thus the bound on $p$ is optimal up to the value of $\varepsilon$ and this answers a question of Bedenknecht, Han, Kohayakawa, and Mota.
\end{abstract}

\date{\today}

\maketitle

\section{Introduction}
\label{s1}
The study of the existence of Hamilton cycles is one of the most classical problems in graph theory.
A celebrated result of Dirac~\cite{Dirac} from 1952 states that every graph on $n$ $(n\geq 3)$ vertices with minimum degree at least $n/2$ contains a Hamilton cycle.
After 20 years, Karp~\cite{Karp} showed that it is NP-complete to determine whether a graph has a Hamilton cycle.
During the past two decades there has been a focus on extending Dirac's theorem to hypergraphs.

Let $k\geq 2$, a \emph{$k$-uniform hypergraph} (for short, \emph{$k$-graph}) $H=(V,E)$ consists of a vertex set $V$ of order $n$ and an edge set $E$, where $E$ is a family of $k$-subsets of $V$, that is, $E\subseteq \binom{V}{k}$.
If $E= \binom{V}{k}$, then $H$ is a \emph{complete $k$-graph}, denoted by $K_n^{(k)}$.
For $k=2$, it is an ordinary graph.
For any $d$-subset $S\subseteq V$ with $1\leq d\leq k-1$, the \emph{degree} of $S$, denoted by $\deg_H(S)$, is the number of edges of $E$ containing $S$, that is, $\deg_H(S)=\left|\left\{e\in E\colon\, e\supseteq S\right\}\right|$.
The \emph{minimum $d$-degree} $\delta_{d}(H)$ of $H$ is the minimum of $\deg_H(S)$ over all $d $-subsets $S$ of $V$.
We often call $\delta_{k-1}(H)$ the \emph{minimum codegree} of $H$.
For two $k$-graphs $G$ and $H$, denote by $G\cup H$ (or $G\cap H$) the $k$-graph with vertex set $V(G)\cup V(H)$ (or $V(G)\cap V(H)$) and edge set $E(G)\cup E(H)$ (or $E(G)\cap E(H)$).
\subsection{$\ell$-cycles}
Given $0<\ell<k$, an \emph{$\ell$-cycle} is a $k$-graph whose vertices can be ordered cyclically such that every edge consists of $k$ consecutive vertices and every pair of consecutive edges (in the natural ordering of the edges) intersects in exactly $\ell$ vertices.
A \emph{Hamilton $\ell$-cycle} is an $\ell$-cycle which is a spanning subgraph.
In particular, a Hamilton $(k-1)$-cycle is usually called a \emph{tight Hamilton cycle} and a Hamilton $1$-cycle is usually called a \emph{loose Hamilton cycle}.

In 1999, Katona and Kierstead~\cite{Katona} first extended Dirac's theorem to $k$-graphs and conjectured that every $k$-graph $H$ on $n\geq k+1\geq 4$ vertices with $\delta_{k-1}(H)\geq \left\lfloor \frac{n-k+3}{2}\right\rfloor$ has a tight Hamilton cycle.
The above conjecture, if true, is best possible (see~\cite{Katona}).
Later R\"{o}dl, Ruci\'{n}ski, and Szemer\'{e}di~\cite{Rodl3} confirmed the conjecture for $k=3$ and sufficiently large $n$ (we assume that $n$ is sufficiently large in this paper unless stated otherwise).
The same authors~\cite{Rodl2} showed that $\delta_{k-1}(H)\geq (1/2+o(1))n$ guarantees a tight Hamilton cycle for all $k\ge 3$.
This inspired a large amount of research on determining the minimum $d$-degree conditions that force Hamilton $\ell$-cycles in $k$-graphs, for $1\leq d,\ell<k$, see~\cite{Bastos1,Bastos2,Bub,Czyg,Gleb,Hiep,Han2,Han3,Keev,Kuhn3,Kuhn1,Reih,Rodl4,Rodl1} and we recommend to the reader the surveys~\cite{Rodl5,Zhao} for a detailed discussion on this topic.

The existence of Hamilton $\ell$-cycles has also been considered in the binomial random $k$-graph $G^{(k)}(n,p)$, which contains $n$ vertices and each $k$-tuple forms an edge independently with probability $p$.
The threshold for the existence of Hamilton cycles in $G(n,p)$ is about $(\log n)/n$ proven by P\'{o}sa~\cite{Posa1} and Kor\v{s}hunov~\cite{Kors} independently.
The threshold for the existence of Hamilton $\ell$-cycles has been studied by Dudek and Frieze~\cite{Dudek2,Dudek1}, who proved that for $\ell= 1$ the threshold is $(\log n)/n^{k-1}$, and for $\ell\geq 2$ the threshold is $1/n^{k-\ell}$.
Recently, Narayanan and Schacht \cite{NS} determined the sharp threshold for $\ell\geq 2$, which resolved several questions raised by Dudek and Frieze \cite{Dudek2}.
\subsection{Powers of $\ell$-cycles}
Powers of graphs are natural objects to study.
Given $k\geq 2$ and $r\geq 1$, we say that a $k$-graph is an \emph{$r^{th}$ power of a tight cycle} (for short, a \emph{$(k,r)$-cycle}) if its vertices can be ordered cyclically so that each consecutive $k+r-1$ vertices span a copy of $K^{(k)}_{k+r-1}$ and there are no other edges than the ones forced by this condition.
This extends the notion of (tight) cycles in hypergraphs, which corresponds to the case $r = 1$.
The P\'{o}sa-Seymour conjecture~\cite{Erdos64,Sey74} asserts that every graph on $n$ vertices
with minimum degree at least $rn/(r+1)$ contains the $r^{th}$ power of a Hamilton cycle.
This conjecture can also be seen as a generalization of Dirac's theorem and was solved for large $n$ by Koml\'{o}s, S\'{a}rk\"{o}zy, and Szemer\'{e}di \cite{Komlos}.
For the existence of the $r^{th}$ power of a tight Hamilton cycle in $k$-graphs, not much is known in general.
Bedenknecht and Reiher~\cite{BR20} showed that any $3$-graph on $n$ vertices with codegree at least $(4/5 + o(1))n$ contains the square of a tight Hamilton cycle, which was extended in~\cite{PSS21} to arbitrary $k\ge 2$ and arbitrary $r\ge 2$.

Riordan~\cite{Rior} first studied powers of Hamilton cycle in $G(n,p)$ and obtained that the threshold for the existence of the $r^{th}$ power of a Hamilton cycle is $n^{-1/r}$ when $r\ge 3$.
In fact, he obtained a more general result, see~\cite{Rior} for more details.
For the case $r=2$, K\"{u}hn and Othus~\cite{Kuhn2} proved that if $p\geq n^{-1/2+\varepsilon}$, then asymptotically almost surely\footnote{We say that an event happens \emph{asymptotically almost surely}, or \emph{a.a.s.}~for short, if the probability that it happens tends to $1$ as $n$ tends to infinity.}, $G(n,p)$ contains the square of a Hamilton cycle for any constant $\varepsilon > 0$.
Recently, Nenadov and \v{S}kori\'{c}~\cite{Nena} gave an improved bound, and the threshold is finally settled by Kahn, Narayanan, and Park~\cite{KNP20}.
For $k\geq 3$ and $r\geq 2$, Parczyk and Person~\cite[Theorem 3.7]{Parc} proved that the threshold for the existence of the $r^{th}$ power of a tight Hamilton cycle in $G^{(k)}(n, p)$ is $n^{-\binom{k+r-2}{k-1}^{-1}}$.
\subsection{Randomly perturbed $k$-graphs}
In this paper we consider the notion of randomly perturbed hypergraphs, which asks how many random edges are required to add to a dense hypergraph in order to make the resulting hypergraph contain a certain structure with high probability.
In 2003, Bohman, Frieze, and Martin~\cite{Bohman} showed that adding linearly many random edges to a graph on $n$ vertices with minimum degree at least $\alpha n$ $(\alpha>0)$ will ensure the resulting graph a.a.s.~is Hamiltonian.
The unbalanced complete bipartite graph $K_{\alpha n, (1-\alpha)n}$ shows that linearly many random edges are necessary.
Other results on randomly perturbed $k$-graphs can be found in~\cite{ADRRS21,Balogh,Bott1,Dudek3,Han4,Joos,Kriv1,Kriv2,Mcdo}.
In particular, powers of Hamilton cycles have been studied by Dudek et al.~\cite{ADRRS21, Dudek3} and Nenadov and Truji\'c~\cite{NT21}, where they focused on the case where linearly many random edges are added to a given dense graph.
On the other hand, one may consider the case where a vanishing degree condition is imposed and see how much one could possibly save on the threshold probability.
This is achieved for Hamilton $\ell$-cycles in randomly perturbed $k$-graphs by Han and Zhao~\cite{Han4} (see also~\cite{Mcdo}).
For $r^{th}$ power of tight Hamilton cycles with $r\ge 2$, this was studied by Bedenknecht et al.~\cite{Beden} in randomly perturbed $k$-graphs.
They stated the following question and proved the result for $\alpha>1-\binom{k+r-2}{k-1}^{-1}$.
\begin{question}[\cite{Beden}, Question 5.1]
Let integers $k\geq 3$, $r\geq 2$ and $\alpha> 0$ be given.
Is there $\varepsilon> 0$ such that, if $H$ is a $k$-graph on $n$ vertices with $\delta_{k-1}(H)\geq \alpha n$ and $p = p(n)\geq n^{-\binom{k+r-2}{k-1}^{-1}-\varepsilon}$, then a.a.s.~$H\cup G^{(k)}(n, p)$ contains the $r^{th}$ power of a tight Hamilton cycle?
\end{question}
In this paper, we give a positive answer to this question.
\begin{theorem}[Main result]\label{main}
For $k\geq 3$, $r\geq 2$ and $\alpha>0$, there exists $\varepsilon>0$ such that the following holds.
Suppose $H$ is an $n$-vertex $k$-graph with $\delta_{k-1}(H)\geq \alpha n$ and $p=p(n)\geq n^{-{\binom{k+r-2}{k-1}}^{-1}-\varepsilon}$.
Then a.a.s.~the union $H\cup G^{(k)}(n,p)$ contains the $r^{th}$ power of a tight Hamilton cycle.
\end{theorem}

We remark that Theorem \ref{main} holds for every $k+r\geq 4$ (however, one can not hope for the $-\varepsilon$ term in the condition on $p$ for $k=2,r=1$).
In fact, the case $k=2$, $r\geq 2$ was proved in \cite{Bott2}, and the case $k\geq 3$, $r=1$ was proved in~\cite{Han4} and~\cite{Mcdo} independently.

The bound on $p$ in Theorem~\ref{main} is optimal up to the value of $\varepsilon$ in the following sense.
Indeed, given $k\ge 3$ and $r\ge 2$, let $\alpha>0$ be sufficiently small and $n$ be sufficiently large.
Consider the $n$-vertex $k$-graph $H_0$ whose vertex set $V$ is partitioned into two disjoint sets $A\cup B$ such that $|A|=\alpha n$, $|B|=(1-\alpha) n$, and whose edges are all the $k$-sets in $V$ that intersect $A$.
Clearly $\delta_{k-1}(H_0)= \alpha n$.
We are going to show that there exists $\varepsilon=\varepsilon(\alpha)$ such that if $p\le n^{-\binom{k+r-2}{k-1}^{-1}-\varepsilon}$ then a.a.s.~$H_0\cup G^{(k)}(n,p)$ does not contain the $r^{th}$ power of a tight Hamilton cycle.
Suppose a.a.s.~$H_0\cup G^{(k)}(n,p)$ contains the $r^{th}$ power of a tight Hamilton cycle, denoted by $\mathcal{C}$.
Since $|A| =\alpha n$, $\mathcal{C}$ contains at least $1/\alpha-1$ consecutive vertices in $B$.
Let $m =\left\lfloor \frac{1}{\alpha} -1-\frac{(k-1)(k+r-1)}{k}\right\rfloor\binom{k+r-2}{k-1}$.
Since $B$ is an independent set in $H_0$, this implies that $G^{(k)}(n,p)$ a.a.s.~contains the $r^{th}$ power of a tight path with $1/\alpha-1$ vertices and thus $m$ edges (the definition is given in Section~\ref{s2}).
Notice that $1/m\le (1+\beta) \binom{k+r-2}{k-1}^{-1}\alpha$ for some small $\beta>0$, since $\alpha$ is sufficiently small.
Taking $\varepsilon=(1+\beta)(k+r-1)\binom{k+r-2}{k-1}^{-1}\alpha$, we have
$\varepsilon\ge (k+r-1)/m$, and thus $p\le n^{-\binom{k+r-2}{k-1}^{-1}-\varepsilon}\le n^{-\binom{k+r-2}{k-1}^{-1}-(k+r-1)/m}$.
Let $X$ be the number of copies of such tight path in $G^{(k)}(n,p)$.
Then we have $\mathbb{E}[X]< n^{1/\alpha-1}p^m = o(1)$.
By Markov's inequality (see e.g.~\cite[inequality (1.3)]{Janson}), a.a.s.~$G^{(k)}(n,p)$ contains no $r^{th}$ power of a tight path with $1/\alpha-1$ vertices and $m$ edges, which leads to a contradiction.
In particular, the construction shows that in Theorem~\ref{main}, one cannot take $\varepsilon > (1+o(1))(k+r-1)\binom{k+r-2}{k-1}^{-1}\alpha$ (while our proof of Theorem~\ref{main} only gives a rather weak dependence which we make no attempt to optimize).

In view of the result of Parczyk and Person~\cite{Parc}, Theorem~\ref{main} shows a saving of a polynomial factor $n^{\varepsilon}$ in the threshold probability compared with the purely random model $G^{(k)}(n, p)$.
This polynomial-saving also appears in~\cite{Beden} (although a higher minimum degree condition is imposed) and~\cite{Bott2}.
In contrast, in some other results, e.g.~Bohman et al.~\cite{Bohman}, the saving is only a poly-logarithm factor.
It seems that this interesting dichotomy can be explained in the following way.
Suppose we are looking for the existence of a power of a Hamilton cycle.
In our model, the dense $k$-graph $H$ can help on embedding some vertices to the cycle, and then leaves the job of embedding power of (tight) paths of constant length.
Informally, it is reduced to finding a collection of vertex-disjoint constant-length paths whose union covers $(1-o(1))n$ vertices.
Clearly constant-length paths should be ``cheaper'' than cycles and thus we expect to observe a saving of polynomial factor on the threshold probability $p$.
However, such a saving should not be expected when $p$ is too low, e.g., when $p=\log(n)/n$ (for Hamiltonicity), getting $p'=n^{-1-o(1)}$ results in only sublinearly many random edges, which is clearly not enough by considering the above construction.

The rest of the paper is organized as follows.
In Section~\ref{s2} we introduce some notation and main lemmas for the proof of Theorem~\ref{main} and deduce the theorem based on the lemmas.
In Section~\ref{s3} we prove some preliminary results in random hypergraphs that will be used in the proof of our connecting lemma and absorbing lemma.
In Section~\ref{s4} we prove our reservoir lemma and path-cover lemma, and in Section~\ref{s5} we show our connecting lemma and absorbing lemma.
Throughout the paper, we omit floor and ceiling functions unless it is necessary.
We write $\alpha\ll \beta\ll \gamma$ to mean that we can choose positive constants $\alpha,\beta,\gamma$ from right to left.
More precisely, there are two monotone functions $f$ and $g$ such that, given $\gamma$, whenever $\beta\le f(\gamma)$ and $\alpha\le g(\beta)$, the subsequent statements hold.
Hierarchies with more constants are defined analogously.
Implicitly, we assume that all constants appearing in a hierarchy are positive.
\section{Outline of the Proof of Theorem~\ref{main}}
\label{s2}
\subsection{Notation}
In this subsection we introduce some notation used throughout the paper.
Given a $k$-graph $H=(V,E)$, we use $v_H$ and $e_H$ to denote the number of vertices and edges of $H$, respectively.
For a vertex $v\in V$, the \emph{link} $L_v$ of $v$ is the $(k-1)$-graph with vertex set $V(L_v)=V\setminus\{v\}$ and edge set $E(L_v)=\left\{S\colon\,S\cup \{v\}\in E\right\}$.
Given two vertex sets $S$ and $R$ such that $|S|=s<k$, we denote by $\deg_H(S,R)$ the number of $(k-s)$-sets $T\subseteq R$ such that $S\cup T$ is an edge of $H$, that is, $\deg_H(S,R)=\left|\left\{T\in \binom{R}{k-s}\colon\, S\cup T\in E\right\}\right|$.
Given a positive integer $n$, let $[n]=\{1,2,\ldots,n\}$.

Given $k\geq 2$ and $r\geq 1$, we say that a $k$-graph is an \emph{$r^{th}$ power of a $k$-uniform tight path} (for short, a \emph{$(k,r)$-path}) if its vertices can be ordered such that each consecutive $k+r-1$ vertices span a copy of $K^{(k)}_{k+r-1}$ and there are no other edges than the ones forced by this condition.
For $b\geq k+r-1$, we use $P_b^{k,r}$ to denote a $(k,r)$-path on $b$ vertices.
The notion of a $k$-uniform tight path corresponds to the case $r = 1$.
Moreover, the \emph{ends} of a $(k,r)$-path are its first and last $k+r-2$ vertices (with the order in the $(k,r)$-path), and the vertices of the $(k,r)$-path not belonging to its ends are called \emph{internal} vertices.
Note that the number of edges of $P_b^{k,r}$ can be counted by fixing a clique of size $k+r-1$, and adding vertices one by one following the order of the path.
So we have that
\[
\left|E(P_b^{k,r})\right|= \binom{k+r-1}{k}+\left(b-(k+r-1)\right)\binom{k+r-2}{k-1},
\]
and for $k+r-1\leq t\leq b$, any subgraph of $P_b^{k,r}$ on $t$ vertices has at most $\binom{k+r-1}{k}+\left(t-(k+r-1)\right)\binom{k+r-2}{k-1}$ edges.
For convenience, define
\begin{align}
g_{k,r}(b)&:=\binom{k+r-1}{k}+\left(b-(k+r-1)\right)\binom{k+r-2}{k-1} \label{equ1}\\
           &=\left(b-\frac{(k-1)(k+r-1)}{k}\right)\binom{k+r-2}{k-1} \label{equ2},
\end{align}
then $g_{k,r}(b)$ counts the number of edges in $P_{b}^{k,r}$ (we omit the subscripts $k$ and $r$ in $g_{k,r}(b)$ when there is no danger of confusion).

To simplify notation, throughout the rest of the paper, we write
$$h=k+r-2 \mbox{ and } \sigma=\binom{k+r-2}{k-1}^{-1}.$$
\subsection{New ideas in the proof}
The general framework for building the (power of) Hamilton cycle is the absorption method, which has been a powerful tool in finding spanning subgraphs since R\"{o}dl, Ruci\'{n}ski, and Szemer\'{e}di in~\cite{Rodl1} applied it to establish a generalization of Dirac's theorem to $k$-graphs.
Following this framework, one shall first fix a short path $P$ which can absorb \emph{any} small set of vertices to the interior of the path.
Then the job is reduced to fining a collection of small number ($o(n)$) of vertex-disjoint paths whose union covers almost the entire vertex set and then connecting these paths and $P$ to a cycle.
At the last the vertices outside the cycle will be absorbed by $P$.
We use this framework and also employ Janson's inequality to derive local structures (e.g.~for connecting two paths by a constant number of vertices) using random edges.
However, with these techniques, the authors of~\cite{Beden} were only able to show Theorem~\ref{main} under the stronger condition $\alpha>1-\binom{k+r-2}{k-1}^{-1}$.

When using Janson's inequality, the authors of~\cite{Beden} chose to keep ``root'' vertices free of random edges.
For example, to connect two given tight paths, they would like to first extend the paths with deterministic edges, and then argue that there are many choices of such extensions enforced by the minimum codegree condition.
Then among all those extensions using deterministic edges (e.g., edges of $H$), one can hook up the paths using random edges by Janson's inequality.
Since our aim is to eliminate the codegree requirement almost entirely, under the same framework of absorption (and for $r\ge 2$), we \emph{have to} use random edges even in the initial extensions.
However, this gives a significant challenge as such extensions use random edges containing ``fixed'' vertices.
In the numerical aspect, all applications of Janson's inequality in~\cite{Beden} were helpful only if $\Phi_F=\Phi_F(n,p)\ge Cn$ (where $F$ being the $k$-graph as the connector, see Definition~\ref{def2}, and the definition of $\Phi_F$ is given in Section~\ref{s3}) for some large constant $C$, and such extensions using random edges will cause $\Phi_F$ to be sublinear.

To see why $\Phi_F\ge Cn$ is helpful, consider the following process of finding a family of vertex-disjoint copies of $F$ which together covers all but $\gamma n$ vertices, for some small $\gamma$ and large $n$.
When $\Phi_F\ge Cn$, by Janson's inequality, we obtain that the probability of any vertex set of size $\gamma n$ not containing a copy of $F$ is at most $e^{-n}$, which allows us to apply the union bound to all vertex subsets, so that a.a.s.~this is true for all vertex subsets of size $\gamma n$ \emph{simultaneously}.
Owing to this property, one can find the desired copies of $F$ greedily.

To overcome this issue, we use a recent noval embedding scheme (see Proposition~\ref{p1}), which has been developed in~\cite{Han}.
For example, suppose $\Phi_F=n^{0.1}$ and there are $\Omega(n)$ pairs of paths that we need to connect.
An application of Janson's inequality does not guarantee \emph{simultaneous} connections of these paths because the probability is too low to apply a union bound to avoid clashes, i.e., connections are possible but one cannot guarantee that all connections are done by disjoint paths.
However, a natural greedy attempt shows that all but $n^{0.9}\log n$ of them can be connected.
Now suppose we can ``reset'' the candidates for connection with the $n^{0.9}\log n$ pairs of paths, the same embedding scheme will reduce the leftover pairs of paths to be $n^{0.8}\log^2 n$.
That is, we can finish the embedding after 11 rounds.
This serves as the heart of our embedding scheme to find simultaneous connections of multiple pairs of paths.
It is not hard to reset the candidate pool: we use the probabilistic method to save the codegree condition into multiple blocks, and use the \emph{multi-round exposure} approach to save the random edges in order to expose them in multiple rounds.
\subsection{Proof of Theorem~\ref{main}}
Following previous work~\cite{Hiep,Kuhn3,Rodl3,Rodl1,Rodl2}, we prove Theorem~\ref{main} by using the ``absorption technique'' pioneered by R\"{o}dl, Ruci\'{n}ski, and Szemer\'{e}di~\cite{Rodl1}.
More precisely, we find the desired power of Hamilton cycle by applying the Connecting Lemma (Lemma~\ref{lem:con}), the Absorbing Lemma (Lemma~\ref{lem:abs}), the Reservoir Lemma (Lemma~\ref{lem:res}), and the Path-cover Lemma (Lemma~\ref{lem:cov}).
In order to construct an almost spanning $(k,r)$-cycle of a $k$-graph $H$, we first find some $(k,r)$-paths and then connect them to a $(k,r)$-cycle.
Recall that the ends of a $(k,r)$-path are two ordered $h$-sets that consist of the first and last $h$ vertices, each spanning a copy of $K_{h}^{(k)}$, and the vertices of the $(k,r)$-path not belonging to its ends are its internal vertices.
For a collection of $2t$ mutually disjoint sets of $h$ vertices $A_i,B_i$ we say that a set of $(k,r)$-paths $\mathcal{T}=\{\mathcal{T}_1,\ldots,\mathcal{T}_t\}$ \emph{connects} $(A_i, B_i)_{i\in[t]}$ if all paths are vertex-disjoint and $A_i$ and $B_i$ are the ends of $\mathcal{T}_i$, for all $i\in [t]$.
The connections for a given collection of disjoint $(k,r)$-paths are given by the following lemma (Lemma~\ref{lem:con}).
In addition the lemma allows to restrict the edges used for the connection to a given ``well-connected'' subset (reservoir) $R$ of vertices, and it also allows to find $(k,r)$-paths for the connection whose number of internal vertices are under control.

We start with the reservoir lemma.
In the proof of Theorem~\ref{main}, note that the family of paths that we need to connect occupy essentially the entire vertex set.
To guarantee the connection, we put aside a small set $R$ of vertices aside at the beginning, which is called the \emph{reservoir} and has the property that every $(k-1)$-set has several neighbours in $R$.

\begin{lemma}[Reservoir Lemma]\label{lem:res}
For all $\alpha>0$ and $\gamma>0$, the following holds for sufficiently large $n$.
Let $H=(V,E)$ be an $n$-vertex $k$-graph with $\delta_{k-1}(H)\ge \alpha n$.
Then there is a set $R\subseteq V$ of size $\gamma^2 n/2\le |R|\le 2\gamma^2 n$ such that for all $(k-1)$-sets $S\in\binom{V}{k-1}$ we have $\deg_H(S,R)\ge \alpha |R|/2$.
\end{lemma}

The following is our connecting lemma, which allows us to connect a collection of $(k,r)$-paths to a $(k,r)$-cycle.

\begin{lemma}[Connecting Lemma]\label{lem:con}
Let $k\ge 3$, $r\ge 2$, and $b\geq (k+r)^2\binom{k+r-2}{k-1}$ be an even integer.
Suppose $1/n\ll \varepsilon,\gamma \ll \alpha,1/k,1/r,1/b$.
Let $H=(V,E)$ be an $n$-vertex $k$-graph and $R\subseteq V$, and suppose $p=p(n)\geq n^{-\sigma-\varepsilon}$.
For every collection of $2t$ with $t\le \gamma n$ mutually disjoint ordered $h$-sets $A_i$ and $B_i$ in $V$, each spanning a copy of $K_h^{(k)}$ in $H$, the following holds for $V'=\bigcup_{i\in [t]}(A_i\cup B_i)$.

If $\deg_{H}(S,R)\geq \alpha n$ for all $S\in \binom{V}{k-1}$, then a.a.s.~$H\cup G^{(k)}(n,p)$ contains a family of $(k,r)$-paths $\mathcal{T}=\{\mathcal{T}_1,\ldots, \mathcal{T}_t\}$, each with exactly $b$ internal vertices in $R$, connecting $(A_i, B_i)_{i\in [t]}$, which contains vertices from $V'\cup R$ only.
\end{lemma}

The next result is our absorbing lemma.
Such a lemma usually outputs an absorbing path, which can ``absorb'' any small set of vertices.
We prove the following weaker version, in which our absorbing path can only absorb subsets of a given vertex set.
We shall see that this weaker version is easier to prove and in fact suffices for our purpose.

\begin{lemma}[Absorbing Lemma]\label{lem:abs}
Let $k\ge 3$ and $r\ge 2$, and suppose $1/n\ll \varepsilon,\gamma\ll\alpha,1/k,1/r$.
Let $H=(V,E)$ be an $n$-vertex $k$-graph and $X\subseteq V$ with $|X|\le \gamma n$, and suppose $p=p(n)\geq n^{-\sigma-\varepsilon}$.
Then a.a.s.~$H'=H\cup G^{(k)}(n,p)$ has the following property.

If $\deg_{H}(S,V\setminus X)\geq \alpha n$ for all $S\in \binom{V}{k-1}$, then the $k$-graph $H'[V\setminus X]$ contains a $(k,r)$-path $P_{abs}$ with $|V(P_{abs})|\le\sqrt{\gamma} n$ such that for every $U\subseteq X$ there exists a $(k,r)$-path $Q$ in $H'$ with $V(Q)=V(P_{abs})\cup U$ having the same ends as $P_{abs}$.
\end{lemma}

The last lemma below allows us to find ``not too many'' $(k,r)$-paths to cover all but a small fraction of vertices.

\begin{lemma}[Path-cover Lemma]\label{lem:cov}
Let $k\ge 3$ and $r\ge 2$, and suppose $1/n\ll \varepsilon\ll1/m\ll \gamma,1/k,1/r$.
Let $G^{(k)}(n,p)$ be the binomial random $k$-graph on an $n$-vertex set $V$.
If $p=p(n)\geq n^{-\sigma-\varepsilon}$, then a.a.s.~$G^{(k)}(n,p)$ has the following property.

For every subset $Y\subseteq V$ there exists a collection of at most $\gamma n/2$ vertex-disjoint $(k,r)$-paths, each with exactly $m$ vertices, covering all but at most $\gamma n$ vertices of $V\setminus Y$.
\end{lemma}

Now we can prove Theorem~\ref{main} by combining the four lemmas.

\begin{proof}[Proof of Theorem~\ref{main}]
Let $k\ge 3$, $r\ge 2$ and $\alpha> 0$ be given and let $b=4(k+r)^2/\sigma$.
Suppose $1/n\ll\varepsilon\ll \gamma\ll \alpha,1/k,1/r$.
Let $H=(V,E)$ be an $n$-vertex $k$-graph with $\delta_{k-1}(H)\geq \alpha n$ and suppose $p=p(n)\geq n^{-\sigma-\varepsilon}$.
We will expose $G=G^{(k)}(n,p)$ in four rounds: $G=G_1\cup G_2\cup G_3\cup G_4$ with $G_1, G_2, G_3$ and $G_4$ as independent copies of $G^{(k)}(n,p')$, where $(1-p')^4=1-p$.
Note that $(1-p')^4>1-4p'$, so $p'> p/4\geq n^{-\sigma-2\varepsilon}$.

By Lemma~\ref{lem:res} there is a reservoir set $R\subseteq V$ of size $\gamma^2 n/2\le|R|\leq 2\gamma^2 n$ such that $\deg_{H}(S,R)\geq\alpha |R|/2$ for all $(k-1)$-subsets $S\in \binom{V}{k-1}$.

Since $\delta_{k-1}(H)\ge \alpha n$ and $\gamma\ll \alpha$, we have $\deg_{H}(S,V\setminus R)\geq \alpha n-2\gamma^2n\ge \alpha n/2$ for $S\in \binom{V}{k-1}$.
By Lemma~\ref{lem:abs} with $(R,\alpha/2,2\gamma^2,2\varepsilon)$ playing the role of $(X,\alpha,\gamma,\varepsilon)$, a.a.s.~the $k$-graph $H\cup G_1$ contains an absorbing $(k,r)$-path $P_{abs}$ on $V\setminus R$ and let $A_0$ and $B_0$ be the ends of $P_{abs}$.
Then $|V(P_{abs})|\le \sqrt{2\gamma^2} n\le 2\gamma n$ and $P_{abs}$ has the following absorption property:
for every subset $U\subseteq R$, there is a $(k,r)$-path $Q$ in $H\cup G_1$ such that $V(Q)=V(P_{abs})\cup U$ and $Q$ has the ends $A_0$ and $B_0$.

Let $V_0=V\setminus \left(V(P_{abs})\cup R\right)$.
Now we apply Lemma~\ref{lem:cov} to $G_2$ with $(V(P_{abs})\cup R,\gamma^{12},2\varepsilon)$ playing the role of $(Y,\gamma,\varepsilon)$, and obtain a collection of disjoint $(k,r)$-paths $P_1,\ldots,P_t$ with $t\le \gamma^{12} n/2$ that cover the set $V_0$ except for a small subset $U\subseteq V_0$ with $|U|\le\gamma^{12} n$.

Let $H'=H\cup G_3$ and $V_1=R\cup U$, and let $n_1=|V_1|$.
Next we find a $(k,r)$-path in $H'[V_1]$ containing all vertices of $U$.
Note that $\gamma^2 n/2\le |R|\le n_1=|R|+|U|\le 3\gamma^2n$, and then $|U|\le \gamma^{12} n\le \gamma^9 n_1$ and $n_1\le 6|R|$.
Thus $\deg_{H}(S,R)\geq\alpha |R|/2\ge \alpha n_1/12$ for all $S\in \binom{V_1}{k-1}$.
Now we apply Lemma ~\ref{lem:abs} to $H'[V_1]$ with $(U,n_1,\alpha/12,\gamma^{12},2\varepsilon)$ playing the role of $(X,n,\alpha,\gamma,\varepsilon)$, and conclude that a.a.s.~the $k$-graph $H'[V_1]$ contains an absorbing $(k,r)$-path $P'_{abs}$ on $R$, whose ends are denoted by $A'_0$ and $B'_0$.
Then $|V(P'_{abs})|\le \gamma^3 n_1\le 3\gamma^5 n$ and there exists a $(k,r)$-path $Q'$ in $H'[V_1]$ with $V(Q')=V(P'_{abs})\cup U$ having the same ends as $P'_{abs}$.

For all $i\in [t]$, we denote the ends of $P_i$ by $A_i$ and $B_i$.
Let $A_{t+1}=A'_0$, $B_{t+1}=B'_0$ and $A_{t+2}=A_0$.
Note that for all $S\in \binom{V}{k-1}$ we have $\deg_{H}(S,R)\geq\alpha |R|/2\ge \alpha\gamma^2 n/4$.
By using Lemma~\ref{lem:con} with $(\alpha\gamma^2/4,t+2,\gamma^{11},2\varepsilon)$ playing the role of $(\alpha,t,\gamma,\varepsilon)$, a.a.s.~the $k$-graph $H\cup G_4$ contains a collection of $(k,r)$-paths, each with exactly $b$ internal vertices in $R$, connecting the family $(B_i,A_{i+1})_{0\le i\le t+1}$.
This implies that we connect the $(k,r)$-paths $P_{abs},P_1,\ldots,P_t,Q'$ to a $(k,r)$-cycle $\mathcal{C}$.

Let $U'=V\setminus V(\mathcal{C})$ be the set of vertices not contained in $\mathcal{C}$, i.e.~the leftover vertices in the reservoir $R$.
Then we can utilize the absorbing property of $P_{abs}$ to obtain a $(k,r)$-path $Q''$ with $V(Q'') = V(P_{abs})\cup U'$ having the same ends as $P_{abs}$.
Therefore, we can replace $P_{abs}$ by $Q''$ in $\mathcal{C}$ and obtain the desired $r^{th}$ power of a tight Hamilton cycle in $H\cup G^{(k)}(n,p)$.
\end{proof}
\section{Preliminaries}
\label{s3}
In this section we discuss some relevant results related to binomial random $k$-graphs $G^{(k)}(n, p)$.
Following~\cite{Han,Janson}, given a $k$-graph $F$, define $\Phi_F=\Phi_F(n,p):=\min\left\{n^{v_H}p^{e_H}\colon\,H\subseteq F,\, e_H>0\right\}$.
Here we will be interested in the appearance of $k$-graphs in $G^{(k)}(n, p)$ where we require some subset of vertices to be already fixed in place.
Therefore, for a $k$-graph $F$, and a subset of \emph{independent} vertices $W \subseteq V (F)$, i.e.~$E(F[W])=\emptyset$, we define
\[
\Phi_{F,W}=\Phi_{F,W}(n,p):=\min\left\{n^{v_H-v_{H[W]}}p^{e_H}\colon\,H\subseteq F,\, e_H>0\right\}.
\]
It is easy to see that $\Phi_F=\Phi_{F,\emptyset}$ and $\Phi_{F\setminus W}\geq \Phi_{F,W}$ for any $F$ and independent set $W\subseteq V (F)$.

We will use the following result proved by Bedenknecht et al.~in~\cite{Beden}.
Recall that $h=k+r-2$ and $\sigma=\binom{k+r-2}{k-1}^{-1}$.

\begin{proposition}[\cite{Beden}, Proposition 2.3]\label{p2}
Let $k\geq 3$, $r\geq 2$, $b\geq k+r-1$ and $C>0$.
Suppose $0<\varepsilon<\min\left\{(2g(b)\right)^{-1},\sigma/3\}$ and $1/n\ll 1/C,1/k,1/r,1/b$.
If $p=p(n)\geq n^{-\sigma-\varepsilon}$, then $\Phi_{P_b^{k,r}}\geq Cn$.
\end{proposition}

Now let us bound $\Phi_{F\setminus W}$ and $\Phi_{F,W}$ for some given $k$-graphs $F$ with an independent subset $W$ of $V(F)$ fixed in place, which will be used to prove our connecting lemma and absorbing lemma.

\begin{definition}\label{def1}
Let $A$ be a labelled $k$-graph whose ordered vertex set is denoted by
$$V(A)=(v_1,\ldots,v_{h},v,v_{h+1},\ldots,v_{2h}),$$
and whose edge set can be partitioned into two disjoint sets $E(A)=E_1\cup E_2$ such that
  \begin{enumerate}[label=\rmlabel]
  \item \label{BHMi} $E_1$ induces a $(k,r)$-path $P_{2h}^{k,r}$ under the order $(v_1,\ldots,v_{h},v_{h+1},\ldots,v_{2h})$;
  \item \label{BHMii} $E_2$ consists of the edges containing $v$ with edges $\left\{vv_j\ldots v_{j+(k-2)}\colon\, j\in[k+2r-2]\right\}$ removed.
\end{enumerate}
\end{definition}

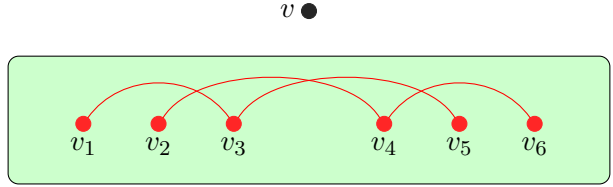
\begin{figure}[h]\label{fig1}
\begin{center}
\begin{tikzpicture}
[inner sep=2pt,
   vertex/.style={circle, draw=red!85, fill=red!85},
   vertex1/.style={circle, draw=black!85, fill=black!85},
   ]
\begin{pgfonlayer}{background}    
\draw[rounded corners, fill=green!20] (-4,-0.8) rectangle (4, 0.9);
\end{pgfonlayer}
\node at (-3,0) [vertex, label=below: $v_1$] {};
\node at (-2,0) [vertex, label=below: $v_2$] {};
\node at (-1,0) [vertex, label=below: $v_3$] {};
\node at (1,0) [vertex, label=below: $v_4$] {};
\node at (2,0) [vertex, label=below: $v_5$] {};
\node at (3,0) [vertex, label=below: $v_6$] {};
\node at (0,1.5) [vertex1, label=left: $v$] {};
\draw[color=red]  (-2.94, 0.08) .. controls (-2.5, 0.7) and (-1.5, 0.7) ..  (-1.06, 0.08); 
\draw[color=red]  (-1.94, 0.08) .. controls (-1.5, 0.8) and (0.5, 0.8) ..  (0.94, 0.08); 
\draw[color=red]  (-0.94, 0.08) .. controls (-0.5, 0.8) and (1.5, 0.8) ..  (1.94, 0.08); 
\draw[color=red]  (1.06, 0.08) .. controls (1.5, 0.7) and (2.5, 0.7) ..  (2.94, 0.08); 
\end{tikzpicture}
\caption{\small An illustration of Definition~\ref{def1} for the case $k=3$ and $r=2$. Red edges represent the edges in the link of $v$ in $A$, and the edge set $E_2=\{vv_1v_3,vv_2v_4,vv_3v_5,vv_4v_6\}$, and the green part represent the edge set $E_1$.}
\end{center}
\end{figure}

\begin{example}\label{ex1}
For $k=3$ and $r=2$, the ordered vertex set $V(A)$ and edge set $E(A)$ defined in Definition~$\ref{def1}$ are as follows: $V(A)=(v_1,v_2,v_3,v,v_4,v_5,v_6)$ and $E(A)=\{v_iv_{i+1}v_{i+2}\colon\,i\in[4]\}\cup\{v_iv_{i+j}v_{i+3}\colon\,i\in[3],j=1,2\}\cup\{vv_iv_{i+2}\colon\,i\in[4]\}$.
\end{example}

In the following lemma we bound $\Phi_A$ with a root vertex $v$ from below.

\begin{lemma}\label{l1}
Let $k\geq 3$, $r\geq 2$ and $C>0$.
Suppose $0<\varepsilon\leq\frac{1}{4g(2k+2r-3)}\min\{\frac{2}{k(k+1)},(k+2)\sigma\}$ and $1/n\ll 1/C,1/k,1/r$.
Let $A$ be the labelled $k$-graph as in Definition~$\ref{def1}$, and set $W=\{v\}$.
If $p=p(n)\geq n^{-\sigma-\varepsilon}$, then $\Phi_{A\setminus W}\geq Cn$ and $\Phi_{A,W}\geq Cn$.
\end{lemma}

\begin{proof}
Let $H$ be a subgraph of $A$ with $v_H$ vertices and $e_H$ edges, where $e_H>0$.
By the definitions of $\Phi_{A\setminus W}$ and $\Phi_{A,W}$, our goal is to prove that $n^{v_H}p^{e_H}\geq Cn$ if $v\notin V(H)$ and $n^{v_H-1}p^{e_H}\geq Cn$ if $v\in V(H)$.
It suffices to assume that $H$ is an induced subgraph of $A$.

Assume that $v_H=t+1$, then we have $k-1\leq t\leq 2h$.
If $v\notin V(H)$, then $H$ is a subgraph of $P_{2h}^{k,r}$.
It is easy to check that the conditions needed for Proposition~\ref{p2} with $b=2h$ are satisfied, so we have $n^{v_H}p^{e_H}\geq Cn$.

Denote by $e_A^v$ the number of edges containing $v$ in $A$.
Note that the link $L_v$ of $v$ in $A$ is a $(k-1,r)$-path under the order $(v_1,\ldots,v_{2h})$ with the edges $\{v_j\ldots v_{j+(k-2)}\colon\, j\in[k+2r-2]\}$ removed.
Thus, from (\ref{equ1}), we obtain
\begin{align}
e_A^v&=g_{k-1,r}(2h)-(k+2r-2)\nonumber \\
  &=\binom{k+r-2}{k-1}+(k+r-2)\binom{k+r-3}{k-2}-(k+2r-2) =  k\binom{k+r-2}{k-1}-(k+2r-2). \nonumber
\end{align}

Next we consider the case $v\in V(H)$.
We have the following two cases depending on the value of $t$.
Note that it suffices to prove $t-\sigma e_H\geq 1+2e_H \varepsilon$.
Indeed, then we have
\[
n^{t}p^{e_H}\geq n^{t}\left(n^{-\sigma-\varepsilon}\right)^{e_H}\geq n^{1+2e_H\varepsilon-e_H\varepsilon} \geq Cn.
\]

\begin{enumerate}[itemindent=1em,topsep=6pt,label={\textbf{Case 1.}}]
\item $h+1 \leq t\leq 2h$.
\end{enumerate}

In this case, the edges of $H$ can be divided into two classes: the edges containing $v$ (there are at most $e_A^v$ such edges), and the edges not containing $v$, which form a $t$-vertex subgraph of $P_{2h}^{k,r}$.
So we have that $e_{H\setminus\{v\}}\leq g(t)$.
Therefore, from (\ref{equ2}), we have
\begin{align}
e_H&\leq g(t)+e_A^v=\left(t-\frac{(k+r-1)(k-1)}{k}\right)\binom{k+r-2}{k-1}+k\binom{k+r-2}{k-1}-(k+2r-2) \nonumber \\
   &=\left(t+1-\frac{(k-1)(r-1)}{k}\right)\binom{k+r-2}{k-1}-(k+2r-2) \nonumber \\
   &\leq \left(t+1-\frac{(k-1)(r-1)}{k}\right)\binom{k+r-2}{k-1}-(k+2), \nonumber
\end{align}
and then
$t-\sigma e_H\geq \frac{(k-1)(r-1)}{k}-1+(k+2)\sigma$.
Let $f=\frac{(k-1)(r-1)}{k}-1+(k+2)\sigma$, and we claim that $f\geq 1+2e_H\varepsilon$.
In fact, $f\geq 1+(k+2)\sigma$ when $r\geq 4$ and $f\geq 1+\frac{2}{k(k+1)}$ when $r=2, 3$.
It suffices to show that $\min\left\{(k+2)\sigma,\frac{2}{k(k+1)}\right\}\geq 2e_H\varepsilon$.
By the choice of $\varepsilon$, we have
\[
\min\left\{(k+2)\sigma,\frac{2}{k(k+1)}\right\}\geq 4g(t+1)\varepsilon\geq 2(g(t)+g(t+1))\varepsilon\geq 2e_H\varepsilon.
\]

\begin{enumerate}[itemindent=1em,topsep=6pt,label={\textbf{Case 2.}}]
\item $k-1\leq t\leq h$.
\end{enumerate}

In this case we have $e_H\leq \binom{t+1}{k}$.
Note that $\binom{t+1}{k}=\binom{k}{k}+\binom{k}{k-1}+\cdots+\binom{t}{k-1}\leq 1+(t-k+1)\binom{t}{k-1}\leq 1+(t-k+1)\binom{k+r-2}{k-1}$.
Therefore, we have
\begin{align}
t-\sigma e_H&\geq t-\sigma-(t-k+1)=k-1-\sigma \geq 2-1/k\geq 5/3\geq 1+2e_H \varepsilon  \nonumber
\end{align}
by the choice of $\varepsilon$.

The proof is completed.
\end{proof}

The following definition is in preparation for connecting many $(k,r)$-paths to a $(k,r)$-cycle.
Specifically, we will connect their ends by using a constant number of vertices.
Let $F$ be a $k$-graph with an ordered vertex set $V(F)$.
Suppose $H$ is a subgraph of $F$, we say that $V(H)$ forms an \emph{interval} if the vertices of $H$ are consecutive following the order of $V(F)$.

\begin{definition}\label{def2}
Let $b\in 2\mathbb{N}$ and $B$ be a labelled $k$-graph whose ordered vertex set is denoted by
\[
V(B)=(w_1,\ldots,w_{h},v_1,\ldots,v_b,u_{h},\ldots,u_1),
\]
and whose edge set $E(B)$ consists of the edges of the $(k,r)$-path on $V(B)$ with the two $k$-uniform tight paths under the order $(w_1,\ldots,w_{h},v_1,\ldots,v_{b/2})$ and the order $(v_{b/2+1},\ldots,v_b,u_{h},\ldots,u_1)$, and all edges induced by $(w_1,\ldots,w_{h})$ and $(u_{h},\ldots,u_1)$ removed.
\end{definition}

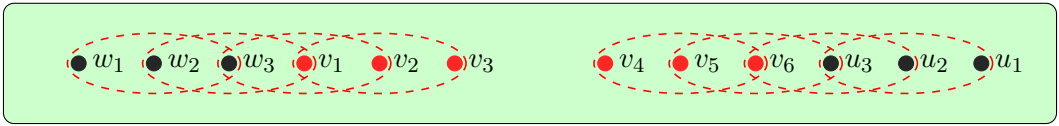
\begin{figure}[h]\label{fig2}
\begin{center}
\begin{tikzpicture}
[inner sep=2pt,
   vertex/.style={circle, draw=red!85, fill=red!85},
   vertex1/.style={circle, draw=black!85, fill=black!85},
   ]
\begin{pgfonlayer}{background}    
\draw[rounded corners, fill=green!20] (-7,-0.8) rectangle (7, 0.8);
\end{pgfonlayer}
\node at (-6,0) [vertex1, label=right: $w_1$] {};
\node at (-5,0) [vertex1, label=right: $w_2$] {};
\node at (-4,0) [vertex1, label=right: $w_3$] {};
\node at (-3,0) [vertex, label=right: $v_1$] {};
\node at (-2,0) [vertex, label=right: $v_2$] {};
\node at (-1,0) [vertex, label=right: $v_3$] {};
\node at (1,0) [vertex, label=right: $v_4$] {};
\node at (2,0) [vertex, label=right: $v_5$] {};
\node at (3,0) [vertex, label=right: $v_6$] {};
\node at (4,0) [vertex1, label=right: $u_3$] {};
\node at (5,0) [vertex1, label=right: $u_2$] {};
\node at (6,0) [vertex1, label=right: $u_1$] {};
\tikzstyle{style 1}=[semithick] 
\draw[color=red, style 1, dashed] (-5,0) ellipse (1.15 and 0.4); 
\draw[color=red, style 1, dashed] (-4,0) ellipse (1.15 and 0.4); 
\draw[color=red, style 1, dashed] (-3,0) ellipse (1.15 and 0.4); 
\draw[color=red, style 1, dashed] (-2,0) ellipse (1.15 and 0.4); 
\draw[color=red, style 1, dashed] (2,0) ellipse (1.15 and 0.4); 
\draw[color=red, style 1, dashed] (3,0) ellipse (1.15 and 0.4); 
\draw[color=red, style 1, dashed] (4,0) ellipse (1.15 and 0.4); 
\draw[color=red, style 1, dashed] (5,0) ellipse (1.15 and 0.4); 
\end{tikzpicture}
\caption{\small An illustration of Definition~\ref{def2} for the case $k=3$, $r=2$ and $b=6$. Dashed triples represent the edges removed from the $(k,r)$-path on $V(B)$ which is indicated by the green part.}
\end{center}
\end{figure}

\begin{example}\label{ex2}
For $k=3$ and $r=2$, the ordered vertex set $V(B)$ and edge set $E(B)$ defined in Definition~$\ref{def2}$ are as follows:
$V(B)=(w_1,w_2,w_3,v_1,\ldots,v_b,u_3,u_2,u_1)$.
To write the edge set clearly, let $v_{i-3}=w_i$ and $v_{b+4-j}=u_j$ for $i,j\in[3]$.
Then $E(B)=\{v_iv_{i+j}v_{i+3}\colon\,-2\leq i\leq b,\,j=1,2\}\cup\{v_{b/2-1}v_{b/2}v_{b/2+1},v_{b/2}v_{b/2+1}v_{b/2+2}\}$.
\end{example}

The next lemma gives a bound on $\Phi_B$ with a set of root vertices.
Note that in the following $\Phi_{B,W}$ might be sublinear in $n$, and thus we will need the novel embedding scheme (Proposition~\ref{p1}) in its proof.

\begin{lemma}\label{l2}
Let $k\ge 3$, $r\ge 2$ and $b\geq (k+r)^2\binom{k+r-2}{k-1}$ be an even integer, and let $C>0$.
Suppose $0<\varepsilon\leq\left(3b\binom{k+r-2}{k-1}^2\right)^{-1}$ and $1/n\ll 1/C,1/k,1/r,1/b$.
Let $B$ be the labelled $k$-graph as in Definition~$\ref{def2}$, and set $W=\{w_1,\ldots,w_h,u_h,\ldots,u_1\}$.
If $p=p(n)\geq n^{-\sigma-\varepsilon}$, then $\Phi_{B\setminus W}\geq Cn$ and $\Phi_{B,W}\geq Cn^{\sigma/2}$.
\end{lemma}

\begin{proof}
Let $H$ be a subgraph of $B$ with $v_H$ vertices and $e_H$ edges, where $e_H>0$.
Our goal is to prove that $n^{v_H}p^{e_H}\geq Cn$ if $V(H)\cap W=\emptyset$ and $n^{v_H-|V(H)\cap W|}p^{e_H}\geq Cn^{\sigma/2}$ if $V(H)\cap W\neq\emptyset$.
It suffices to assume that $H$ is an induced subgraph of $B$.

Assume that $|W\cap V(H)|=i$ and $v_H=t+i$, then we have $1\leq t\leq b$ and $0\leq i\leq 2h$.
If $W\cap V(H)=\emptyset$, then $i=0$ and $H$ is a subgraph of $P_{b}^{k,r}$.
It is easy to check that the conditions needed for Proposition~\ref{p2} are satisfied, so we have $n^{v_H}p^{e_H}\geq Cn$.

Next we consider the case $W\cap V(H)\neq\emptyset$.
Let $W_1=\{w_1,\ldots,w_h\}$, $W_2=\{u_1,\ldots,u_h\}$ and $Q=\{v_{b/2+1},\ldots, v_{b/2+(k-1)}\}$.
It suffices to show that $e_H\leq t\binom{k+r-2}{k-1}-1$, as this implies
\[
n^tp^{e_H}\geq n^{t}\left(n^{-\sigma-\varepsilon}\right)^{e_H}\geq n^{\sigma-e_H\varepsilon}\geq Cn^{\sigma/2}
\]
by the choice of $\varepsilon$.
We consider the following three cases.

\begin{enumerate}[itemindent=1em,topsep=6pt,label={\textbf{Case 1.}}]
\item $V(H)\cap W_1\neq \emptyset$ and $V(H)\cap W_2= \emptyset$.
\end{enumerate}

Since $e_H>0$ and $W_1$ is an independent set, we have $t\geq 1$.
For every $v\in V(H)\setminus W_1$, denote by $e_H^v$ the number of edges in $H$ containing $v$ as the last vertex following the order of $V(B)$.
Note that $e_H= \sum_{v\in V(H)\setminus W_1}e_H^v$.

If $v$ and the $k+r-2$ vertices before it form an interval, then $e_H^v\leq \binom{k+r-2}{k-1}$ for $v\in Q$ and $e_H^v\leq \binom{k+r-2}{k-1}-1$ for $v\in V(H)\setminus(W_1\cup Q)$; otherwise, $e_H^v\leq \binom{k+r-3}{k-1}\leq \binom{k+r-2}{k-1}-1$.

Now we claim that $e_H\leq t\binom{k+r-2}{k-1}-1$, and it suffices to show that there exists a vertex $v\in V(H)\setminus W_1$ such that $e_H^v\leq \binom{k+r-2}{k-1}-1$.
Note that otherwise we have $V(H)\subseteq W_1 \cup Q$.
In this case we have $e_H^{u}\leq \binom{k+r-2}{k-1}-1$, where $u$ is the first vertex of $H$ in $Q$.

\begin{enumerate}[itemindent=1em,topsep=6pt,label={\textbf{Case 2.}}]
\item $V(H)\cap W_1= \emptyset$ and $V(H)\cap W_2\neq \emptyset$.
\end{enumerate}

This case can be treated similarly by the symmetry of $W_1$ and $W_2$.

\begin{enumerate}[itemindent=1em,topsep=6pt,label={\textbf{Case 3.}}]
\item $V(H)\cap W_1\neq \emptyset$ and $V(H)\cap W_2\neq \emptyset$.
\end{enumerate}

Suppose there are at least $k+r-2$ consecutive vertices missing in $H$, then we split $H$ into $H_1$ and $H_2$, where $H_1$ and $H_2$ are the induced subgraphs by the left side and the right side of the interval respectively.
Note that no edge intersects both $V(H_1)$ and $V(H_2)$.
Let $i_j=|V(H_j)\cap W_j|$ for $j=1,2$.
Clearly, we have $i=i_1+i_2$ and $e(H)=e(H_1)+e(H_2)$.
By Case 1 and Case 2, we have $n^{v_{H_j}-i_j}p^{e_{H_j}}\geq Cn^{\sigma/2}$ for $j=1,2$.
So we have
\[
n^tp^{e_H}=n^{v_{H_1}-i_1}p^{e_{H_1}}n^{v_{H_2}-i_2}p^{e_{H_2}}\geq Cn^{\sigma/2}.
\]
Therefore, we may assume that the number of missing vertices in $V(H)\setminus W$ between any two adjacent intervals is at most $k+r-3$.
Thus we have $t\geq b/(k+r-2)-1$.
Now we split the $V(H)$ into $V(H_1)$ and $V(H_2)$ from the vertex $v_{b/2+k}$, and it makes no difference whether the vertex $v_{b/2+k}$ belongs to $V(H_1)$ or $V(H_2)$.
Note that $V(H_2)\cap Q=\emptyset$.
Let $e_{1,2}=\{e\in E(H)\colon\, e\cap V(H_1)\neq \emptyset,\, e\cap V(H_2)\neq \emptyset\}$.
Then $e_H=e_{H_1}+e_{H_2}+e_{1,2}$ and $e_{1,2}\leq (k+r-2)\binom{k+r-2}{k-1}$ (consider adding vertices in $H_2$ one by one).

For every $v\in V(H_1)\setminus W_1$, denote by $e_H^v$ the number of edges in $H_1$ containing $v$ as the last vertex following the order of $V(B)$.
For every $v\in V(H_2)\setminus W_2$, denote by $e_H^v$ the number of edges in $H_2$ containing $v$ as the first vertex following the order of $V(B)$.
Note that $e_{H_1}= \sum_{v\in V(H_1)\setminus W_1}e_H^v$ and $e_{H_2}=\sum_{v\in V(H_2)\setminus W_2}e_H^v$.

Recall that $V(H_2)\cap Q=\emptyset$.
For $v\in V(H_1)\setminus W_1$, if $v$ and the $k+r-2$ vertices before it form an interval, then $e_H^v\leq \binom{k+r-2}{k-1}$ for $v\in Q$ and $e_H^v\leq \binom{k+r-2}{k-1}-1$ for $v\notin Q$; otherwise, $e_H^v\leq \binom{k+r-3}{k-1}\leq \binom{k+r-2}{k-1}-1$.
For $v\in V(H_2)\setminus W_2$, if $v$ and the $k+r-2$ vertices after it form an interval, then $e_H^v\leq \binom{k+r-2}{k-1}-1$; otherwise, $e_H^v\leq \binom{k+r-3}{k-1}\leq \binom{k+r-2}{k-1}-1$.
Let $v_i=|V(H_i)\setminus W_i|$ for $i=1,2$.
Note that $v_1+v_2=t$.
Then we have
\begin{align}
e_{H_1}+e_{H_2} &\leq \left(k-1\right)\binom{k+r-2}{k-1}+\left(v_1-k+1\right)\left(\binom{k+r-2}{k-1}-1\right)+v_2\left(\binom{k+r-2}{k-1}-1\right) \nonumber \\
 & = t\binom{k+r-2}{k-1}-\left(t-k+1\right), \nonumber
\end{align}
and then $e_H=e_{H_1}+e_{H_2}+e_{1,2}\leq \left(t+k+r-2\right)\binom{k+r-2}{k-1}-\left(t-k+1\right)$.

Recall that $t\geq b/(k+r-2)-1$ and $b\geq (k+r)^2\binom{k+r-2}{k-1}$, which implies $t-k+1\geq b/(k+r-2)-k>(k+r)\binom{k+r-2}{k-1}-k$.
So we have
\begin{align}
e_H &\leq (t+k+r-2)\binom{k+r-2}{k-1}-(t-k+1) \nonumber \\
 & \le (t+k+r-2)\binom{k+r-2}{k-1}-(k+r)\binom{k+r-2}{k-1}+k \nonumber \\
  &=t\binom{k+r-2}{k-1}+k-2\binom{k+r-2}{k-1} \leq t\binom{k+r-2}{k-1}-1. \nonumber
\end{align}
The proof is completed.
\end{proof}
\section{Proof of the Reservoir Lemma and the Path-cover Lemma}
\label{s4}
In this section, we prove our reservoir lemma (Lemma~\ref{lem:res}) and path-cover lemma (Lemma~\ref{lem:cov}).
First we focus on the reservoir lemma.
The existence of such a reservoir set is established by a standard probabilistic argument.
Below we restate and prove Lemma~\ref{lem:res}.

\begin{lemmA*}
For all $\alpha>0$ and $\gamma>0$, the following holds for sufficiently large $n$.
Let $H=(V,E)$ be an $n$-vertex $k$-graph with $\delta_{k-1}(H)\ge \alpha n$.
Then there is a set $R\subseteq V$ of size $\gamma^2 n/2\le |R|\le 2\gamma^2 n$ such that for all $(k-1)$-sets $S\in\binom{V}{k-1}$ we have $\deg_H(S,R)\ge \alpha |R|/2$.
\end{lemmA*}

\begin{proof}
Consider a random subset $R\subseteq V$ with elements chosen independently with probability $\gamma^2$.
It is easy to see that a.a.s.~$R$ satisfies the following two properties:
  \begin{enumerate}[label=\rmlabel]
  \item \label{BHMi} $\gamma^2 n/2\le |R| \le 2\gamma^2 n$;
  \item \label{BHMii} $\deg_{H}(S,R)\ge \alpha |R|/2$ for every $S\in\binom{V}{k-1}$.
\end{enumerate}
Indeed, $X = |R|$ is binomially distributed with expectation $\mathbb{E}[X]=\gamma^2 n$, so the first property follows from
Chernoff's inequality (see e.g.~\cite[Theorem 2.1]{Janson}).
Now let $S\in \binom{V}{k-1}$ be an arbitrary set of size $k-1$ and let $X_S = \deg_{H}(S,R)$.
Then $X_S$ is also binomially distributed with expectation $\mathbb{E}[X_S]=\alpha |R|$.
The second property holds by using the union bound combined with Chernoff's inequality.
Consequently, a reservoir set $R$ with all required properties indeed exists.
\end{proof}

Next we restate and show our path-cover lemma.
\begin{LemmA*}
Let $k\ge 3$ and $r\ge 2$, and suppose $1/n\ll \varepsilon\ll1/m\ll \gamma,1/k,1/r$.
Let $G^{(k)}(n,p)$ be the binomial random $k$-graph on an $n$-vertex set $V$.
If $p=p(n)\geq n^{-\sigma-\varepsilon}$, then a.a.s.~$G^{(k)}(n,p)$ has the following property.

For every subset $Y\subseteq V$ there exists a collection of at most $\gamma n/2$ vertex-disjoint $(k,r)$-paths, each with exactly $m$ vertices, covering all but at most $\gamma n$ vertices of $V\setminus Y$.
\end{LemmA*}

Our proof of Lemma~\ref{lem:cov} will use Proposition~\ref{p2} and the following result of~\cite{Beden}, which ensure that $G^{(k)}(n,p)$ can be almost
covered by $(k,r)$-paths.
Recall that for $m\ge k+r-1$, we denote by $P_m^{k,r}$ the $(k,r)$-path with $m$ vertices and $\left(m-\frac{(k-1)(k+r-1)}{k}\right)\binom{k+r-2}{k-1}$ edges.

\begin{lemma}[\cite{Beden}, Lemma 2.2 (i)]\label{l5}
Let $F$ be a labelled $k$-graph with $m$ vertices and $f$ edges.
Let $\gamma>0$ and suppose $1/n\ll1/C\ll \gamma,1/f,1/b$.
Let $G^{(k)}(n,p)$ be the binomial random $k$-graph on an $n$-vertex set $V$.
If $p=p(n)$ is such that $\Phi_{F}\geq Cn$, then a.a.s.~every induced subgraph of $G^{(k)}(n,p)$ of order $\gamma n$ contains a copy of $F$.
\end{lemma}

Now we are ready to prove Lemma~\ref{lem:cov}.

\begin{proof}[Proof of Lemma~\ref{lem:cov}]
Let $k$, $r$ and $\gamma$ be given and fix a subset $Y\subseteq V$.
We define additional constants such that $1/n\ll \varepsilon\ll 1/C \ll 1/m\ll \gamma,1/k,1/r$.
Applying Proposition~\ref{p2} and Lemma~\ref{l5} to $G^{(k)}(n,p)$ with $F=P_m^{k,r}$ and $f=\left(m-\frac{(k-1)(k+r-1)}{k}\right)\binom{k+r-2}{k-1}$, we conclude that a.a.s~every induced subgraph of $G^{(k)}(n,p)$ of order $\gamma n$ contains a copy of $P_m^{k,r}$.
Note that $n/m\le \gamma^2 n$ by the choice of $m$.
Thus we can greedily find at most $\gamma^2 n$ pairwise vertex-disjoint $(k,r)$-paths $P_m^{k,r}$ in $V\setminus Y$ and these paths cover all but at most $\gamma n$ vertices of $V\setminus Y$.
\end{proof}
\section{Proof of the Connecting Lemma and the Absorbing Lemma}
\label{s5}
In this section we focus on the connecting lemma and absorbing lemma.
The following lemma proved by Han, Morris, and Treglown in~\cite{Han} is a careful application of Janson's inequality, which can be used to give a multi-round greedy algorithm that gives vertex-disjoint embedding of $t$ constant-sized $k$-graphs $F_1,\ldots,F_t$, given specified root vertices.
To digest its statment, one may take all $F_i=F$ to be identical, $|W_i|=1$ \footnote{If one takes all $W_i=\emptyset$, then the requirement on $\Phi'$ reduces to requiring $\Phi_F\ge Cn$.} and $\mathcal F_i=\binom Vb$ for all $i$, and $V'=V$.
Then the lemma says under numerical conditions, among any $s$ root vertices, one can accomplish at least one embedding with given root vertices.

\begin{lemma}[\cite{Han}, Lemma 2.8]\label{l3}
Let $n, t=t(n), s=s(n)\in \mathbb{N}$, $0<\beta<1/2$, and $L,b,w,f,k\in \mathbb{N}$ be such that $k\geq 2$, $Lt, sw\leq \beta n/(4b)$ and $\binom{t}{s}\leq 2^n$.
Let $F_1,\ldots,F_t$ be labelled $k$-graphs with distinguished vertex subsets $W_i\subseteq V(F_i)$ such that $|W_i| \leq w$, $|V(F_i\setminus W_i)| = b$, $e(F_i) = f$ and $e(F_i[W_i]) = 0$ for all $i\in [t]$.
Now let $V$ be an $n$-vertex set and let $U_1,\ldots ,U_t\subseteq V$ be labelled vertex subsets with $|U_i| = |W_i|$ for all $i\in [t]$.
Finally, suppose there are families $\mathcal{F}_1, \ldots,\mathcal{F}_t\subseteq \binom{V}{b}$ of labelled vertex sets such that for each $i\in[t]$, $|\mathcal{F}_i|\geq \beta n^b$.

Now suppose that $1 \leq s(n)\leq t(n)$ and $p = p(n)$ are such that
\begin{center}
$s\cdot \Phi\geq \left(\frac{{2^{b+7}b!}}{\beta^2}\right)\min\left\{Lt\log n,n\right\}$ and $\Phi'\geq \left(\frac{{2^{b+7}b!}}{\beta^2}\right)n$,
\end{center}
where $\Phi= \min\left\{\Phi_{F_i,W_i}\colon\, i\in [t]\right\}$ and $\Phi'=\min\left\{\Phi_{F_i\setminus W_i}\colon\,i\in [t]\right\}$ with respect to $p = p(n)$.
Then, a.a.s., for any $V'\subseteq V$, with $|V'|\geq n-Lt$ and any subset $S\subseteq[t]$ such that $|S| = s$ and $U_i\cap U_j =\emptyset$
for $i\neq j\in[s]$, there exists some $i\in S$ such that there is an embedding (which respects labelling) of $F_i$ in $G^{(k)}(n, p)$ on $V$ which maps $W_i$ to $U_i$ and $V(F_i)\setminus W_i$ to a labelled set in $\mathcal{F}_i$ which lies in $V'$.
\end{lemma}

Preparing for the connecting lemma and absorbing lemma, we prove the following proposition, adapted from~\cite[Proposition 6.21]{Han}, which shows how to use Lemma~\ref{l3} repeatedly to embed the desired hypergraphs.

\begin{proposition}\label{p1}
Let $k,b,f,w,\ell,n,t=t(n)\in \mathbb{N}$ be such that $k\ge 2$, and suppose $1/n\ll 1/C,\gamma\ll \beta,1/f,1/b,1/w,1/k,1/\ell$.
Let $F$ be a labelled $k$-graph with base vertex set $W\subseteq V(F)$ such that $|W| =w$, $|V (F)\setminus W|=b$, $e(F[W]) = 0$ and $e(F)= f$.
Further, suppose that $p = p(n)$ is such that $\Phi_{F\setminus W} \geq Cn$ and $\Phi_{F, W}\geq Cn^{1/\ell}$.

Let $V$ be an $n$-vertex set, and let $U_1,\ldots , U_t\subseteq V$ with $t\le \gamma n$ be pairwise disjoint and be such that $|U_i| = |W|$ for each $i\in[t]$.
Suppose that $\mathcal{F}_1,\ldots, \mathcal{F}_t$ are families of ordered $b$-sets on $V$ such that $|\mathcal{F}_i|\geq \beta n^b$.
Then a.a.s.~there exists a collection of embeddings $\phi_1, \ldots , \phi_t$ such that each $\phi_i$ embeds a copy of $F$ into $G^{(k)}(n, p)$ on $V$ with $W$ being mapped to $U_i$ and $V(F)\setminus W$ being mapped to a set in $\mathcal{F}_i$ which is vertex-disjoint with $\bigcup_{i\in[t]}U_i$.
Furthermore, for $i\neq j$ we have $\phi_i(V (F)\setminus W) \cap \phi_j(V (F) \setminus W) = \emptyset$.
\end{proposition}

\begin{proof}
The idea here is to greedily embed $F$ one by one into $G^{(k)}(n,p)$ by repeatedly applying Lemma~\ref{l3} with $F_i=F$ for all $i\in [t]$.
It suffices to assume that $t=\gamma n$.
For the sake of brevity, we say that an embedding $\phi_i$ of $F$ is \emph{valid} if it maps $W$ to $U_i$ and maps $V (F\setminus W)$ to a set in $\mathcal{F}_i$ which is disjoint from $U = \bigcup_{i\in[t]}U_i$ and also disjoint from $\phi_j(V (F\setminus W))$ for all indices $j\in[t]$ which we have already embedded.

Now we expose $G=G^{(k)}(n,p)$ in $\ell+1$ rounds: $G=\bigcup_{j=1}^{\ell+1}G_j$ with each $G_j$ an independent copy of $G^{(k)}(n,p')$, where $(1-p')^{\ell+1}=1-p$.
Note that $p'>p/(\ell+1)$.
By the definitions of $\Phi_{F,W}$ and $\Phi_{F\setminus W}$, we have $\Phi:=\Phi_{F,W}(n,p')\geq C'n^{1/\ell}$ and $\Phi':=\Phi_{F\setminus W}(n,p')\geq C' n$, where $C'=(\ell+1)^{-f}C$.
We will complete our embedding by $\ell+1$ phases.

For $j\in[\ell]$, let $t_j=\gamma^{-(j-2)} n^{1-(j-1)/\ell}(\log n)^{j-1}$ and $s_j=\gamma^{-(j-1)} n^{1-j/\ell}(\log n)^{j}$.
Moreover, let $t_{\ell+1}=\gamma^{-(\ell-1)}(\log n)^{\ell}$ and $s_{\ell+1}=1$.

Let $T=\{1,\ldots, t\}$.
In the $j$th phase we start with $t_j$ indices $T_j\subseteq T$ and let $R_j = T \setminus T_j$.
Note that $R_j$ is the set of indices of $T$ that we have already embedded before starting the $j$th phase.
So we also have some set of already chosen valid embeddings $\{\phi_i\colon\, i\in R_j\}$.
Let $V_j'' = \bigcup_{i\in R_j}\phi_i(V (F_i))\cup U$ and $\mathcal{F}^{(j)}_i = \{S\in \mathcal{F}_i\colon\, S\cap V_j'' = \emptyset\}$.
Then we have that $|V_j''|\leq (b+w)t$ and $|\mathcal{F}^{(j)}_i|\geq \beta n^b-(b+w)tn^{b-1} \geq\left(\beta-(b+w)\gamma\right)n^b\geq \beta n^b/2$ since $\gamma\ll \beta,1/b,1/w$.
Now applying Lemma~\ref{l3} to the sets $\mathcal{F}^{(j)}_i$ such that $i\in T_j$, and where $L=b+w$ and $(t_j, s_j, \beta/2,f,w,b,p')$ plays the role of $(t, s, \beta,f,w,b,p)$.
Notice that since $C'=(\ell+1)^{-f}C$ and $1/C,\gamma\ll \beta,1/b,1/f,1/w,1/\ell$, we have $C'/\gamma\geq 2^{b+9}(b+w)b!/(\beta^2)$.
Hence,
\begin{align}
s_j\cdot\Phi&\geq \gamma^{-(j-1)}n^{1-\frac{j}{\ell}}(\log n)^jC'n^{\frac{1}{\ell}}= \frac{C'}{\gamma^{j-1}}n^{1-\frac{j-1}{\ell}}(\log n)^j \nonumber\\
 &\geq \frac{2^{b+9}(b+w)b!}{\beta^2}\frac{1}{\gamma^{j-2}}n^{1-\frac{j-1}{\ell}}(\log n)^j =\frac{2^{b+7}b!}{(\beta/2)^2}(b+w)\frac{1}{\gamma^{j-2}}n^{1-\frac{j-1}{\ell}}(\log n)^{j} \nonumber\\
 &\geq \frac{2^{b+7}b!}{(\beta/2)^2}\min\left\{(b+w)t_j\log n,n\right\} \nonumber
\end{align}
for $j\in[\ell]$, and
\begin{align}
s_{\ell+1}\cdot\Phi&\geq C'n^{1/\ell} \geq \frac{2^{b+7}b!}{(\beta/2)^2}(b+w)\frac{1}{\gamma^{\ell-1}}(\log n)^{\ell+1}=\frac{2^{b+7}b!}{(\beta/2)^2}\min\left\{(b+w)t_{\ell+1}\log n,n\right\}. \nonumber
\end{align}
Moreover, we have $\Phi'\geq C'n=(\ell+1)^{-f}Cn \geq \left(\frac{2^{b+7}b!}{(\beta/2)^2}\right)n$ because $1/C\ll \beta,1/b,1/f,1/\ell$.
We conclude that a.a.s.~given any set $V_j'$ of at most $Lt_j$ vertices and any set $S_j$ of $s_j$ indices in $T_j$ such that the sets $U_i$ with $i\in S_j$ are pairwise disjoint, there is a valid embedding of $F$ in $G_j$ which avoids $V_j'$.
Now we can initiate with $V_j'=\emptyset$ and repeatedly find indices $i \in T_j$ for which we have a valid embedding $\phi_i$.
We add this embedding to our chosen embeddings, add the vertices of it to $V_j'$ and delete the index $i$ from $T_j$.
The conclusion of Lemma~\ref{l3} asserts that we continue this process until we have $t_{j+1}$, namely $s_j$, indices left in $T_j$ and we can move to the next phase defining $T_{j+1} = T_j$ (or finish if $j = \ell+1$).
In conclusion, we can get the desired $t$ embeddings of $F$.
\end{proof}
\subsection{The connecting lemma}
In this subsection we use Proposition~\ref{p1} and Lemma~\ref{l2} to prove our connecting lemma, which we restate below for convenience.
Given a $k$-graph $H$ and two vertex subsets $S$ and $R$ with $|S|=s<k$, recall that $\deg_H(S,R)$ is the number of $(k-s)$-sets $T\subseteq R$ such that $S\cup T$ is an edge of $H$.

\begin{Lemma*}
Let $k\ge 3$, $r\ge 2$, and $b\geq (k+r)^2\binom{k+r-2}{k-1}$ be an even integer.
Suppose $1/n\ll \varepsilon,\gamma \ll \alpha,1/k,1/r,1/b$.
Let $H=(V,E)$ be an $n$-vertex $k$-graph and $R\subseteq V$, and suppose $p=p(n)\geq n^{-\sigma-\varepsilon}$.
For every collection of $2t$ with $t\le \gamma n$ mutually disjoint ordered $h$-sets $A_i$ and $B_i$ in $V$, each spanning a copy of $K_h^{(k)}$ in $H$, the following holds for $V'=\bigcup_{i\in [t]}(A_i\cup B_i)$.

If $\deg_{H}(S,R)\geq \alpha n$ for all $S\in \binom{V}{k-1}$, then a.a.s.~$H\cup G^{(k)}(n,p)$ contains a family of $(k,r)$-paths $\mathcal{T}=\{\mathcal{T}_1,\ldots, \mathcal{T}_t\}$, each with exactly $b$ internal vertices in $R$, connecting $(A_i, B_i)_{i\in [t]}$, which contains vertices from $V'\cup R$ only.
\end{Lemma*}

\begin{proof}
Given the setup of the statement, we choose $C$ such that $1/n\ll \varepsilon,\gamma, 1/C\ll \alpha,1/k,1/r,1/b$.
For $i\in[t]$, we write $A_i=(w_1^i,\ldots,w_{h}^i)$ and $B_i=(u_{h}^i,\ldots,u_1^i)$.
Our task is to connect $A_i$ and $B_i$ by using exactly $b$ vertices in $R$ for all $i\in [t]$ such that they are pairwise vertex-disjoint and each of them spans a $(k,r)$-path.

We first find candidates for ordered $b$-sets to connect $A_i$ and $B_i$ for each $i\in [t]$.
Since $\deg_{H}(S,R)\geq \alpha n$, we can extend $A_i$ to a tight path with vertices $(w_1^i,\ldots,w_{h}^i,v_1^i,\ldots,v_{b/2}^i)$ using only vertices from $A_i\cup R$ and there are at least $(\alpha n/2)^{b/2}$ choices for the ordered set $(v_1^i,\ldots,v_{b/2}^i)$.
Similarly, we can extend $B_i$ to a tight path $(v_{b/2+1}^i,\ldots, v_b^i,u_{h}^{i},\ldots,u_1^{i})$ using only vertices from $B_i\cup R$ and there are at least $(\alpha n/2)^{b/2}$ choices for the ordered set $(v_{b/2+1}^i,\ldots, v_b^i)$.
So there are at least $(\alpha n/2)^b$ choices for the ordered $b$-sets $(v_1^i,\ldots, v_b^i)$.
For convenience, we use $\mathcal{F}_i$ to denote the collection of such ordered $b$-sets, then we have $|\mathcal{F}_i|\geq(\alpha n/2)^b$ for every $i\in [t]$.

Next we use the edges of $G=G^{(k)}(n,p)$ to obtain the desired copy of $P_{b+2h}^{k,r}$ with base vertex set $A_i\cup B_i$ for each $i\in [t]$.
The idea is to connect $A_i$ and $B_i$ for each $i\in[t]$ by using Proposition~\ref{p1}.
Suppose $B$ and $W$ are the labelled $k$-graph and vertex set $W\subseteq V(B)$ defined in Definition~\ref{def2} with $b$ as above.
Note that if an ordered set $T\in\mathcal{F}_i$ spans a labelled copy of $B$ with ordered sets $A_i$ and $B_i$ fixed, then $T$ connects $A_i$ and $B_i$.
Moreover, we have $\Phi_{B, W}\geq Cn^{\sigma/2}$ and $\Phi_{B\setminus W}\geq Cn$ by Lemma~\ref{l2}.
Now we apply Proposition~\ref{p1} to $F=B$, $W$ and the sets $\mathcal{F}_i$ for every $i\in[t]$ with $t$, $b$, $C$, $\gamma$, $w=2h$, and $\beta=(\alpha/2)^b$.
Taking $U_i=\{w_1^i,\ldots,w_{h}^i,u_{h}^{i},\ldots,u_1^{i}\}$ for all $i\in [t]$, we conclude that a.a.s.~there is a set of embeddings $\phi_1, \ldots , \phi_t$ such that each $\phi_i$ embeds a copy of $B$ into $G^{(k)}(n, p)$ on $V$ with $W$ being mapped to $U_i$ and $V(B)\setminus W$ being mapped to a set in $\mathcal{F}_i$ which is vertex-disjoint with $\bigcup_{i\in[t]}U_i$.
Furthermore, for $i\neq j$ we have $\phi_i(V (B)\setminus W) \cap \phi_j(V (B) \setminus W) = \emptyset$.
Taking $\mathcal{T}_i=\phi_i(V (B)\setminus W)$ for each $i\in [t]$, then $\mathcal{T}=\{\mathcal{T}_1,\ldots,\mathcal{T}_t\}$ connects $(A_i, B_i)_{i\in [t]}$ which contains vertices from $V'\cup R$ only.
\end{proof}
\subsection{The absorbing lemma}
Now we prove our absorbing lemma.
The next tool we will use is supersaturation obtained by Erd\H{o}s and Simonovits~\cite{Erdos}.
We use the equivalent form of supersaturation, which is to find copies of $k$-partite $k$-rgraphs.
Here we use it to find copies of a $k$-uniform tight path.

\begin{lemma}[\cite{Erdos}, Corollary 2]\label{l4}
Let $P^{k}_s$ be a labelled $k$-uniform tight path with $s$ vertices, and suppose $\beta\ll \theta, 1/k,1/s$.
Then every $n$-vertex $k$-graph $H$ with at least $\theta n^{k}$ edges contains at least $\beta n^s$ copies of $P^{k}_s$.
\end{lemma}

We now give the proof of our absorbing lemma, which we restate below, by applying Proposition~\ref{p1} combined with Lemmas~\ref{lem:con},~\ref{l1} and~\ref{l4}.

\begin{lemma*}
Let $k\ge 3$ and $r\ge 2$, and suppose $1/n\ll \varepsilon,\gamma\ll\alpha,1/k,1/r$.
Let $H=(V,E)$ be an $n$-vertex $k$-graph and $X\subseteq V$ with $|X|\le \gamma n$, and suppose $p=p(n)\geq n^{-\sigma-\varepsilon}$.
Then a.a.s.~$H'=H\cup G^{(k)}(n,p)$ has the following property.

If $\deg_{H}(S,V\setminus X)\geq \alpha n$ for all $S\in \binom{V}{k-1}$, then the $k$-graph $H'[V\setminus X]$ contains a $(k,r)$-path $P_{abs}$ with $|V(P_{abs})|\le\sqrt{\gamma} n$ such that for every $U\subseteq X$ there exists a $(k,r)$-path $Q$ in $H'$ with $V(Q)=V(P_{abs})\cup U$ having the same ends as $P_{abs}$.
\end{lemma*}

Before stating our proof for Lemma~\ref{lem:abs}, we first give the definitions of absorbers and absorbing paths.
We call the $(k,r)$-path $P_{abs}$ in Lemma~\ref{lem:abs} an \emph{absorbing path} for $U$.
The absorbers we use later are simply $(k,r)$-paths on $2h$ vertices, and we can insert $v$ into the middle of such a $(k,r)$-path $P$ to create a $(k,r)$-path on $2h+1$ vertices.
In other words, $P$ can absorb $v$.

\begin{definition}
\emph{Let $v$ be a vertex of a $k$-graph $H$.
An ordered $2h$-subset of vertices $(v_1,\ldots,v_{2h})$ is a \emph{$v$-absorber} if $(v_1,\ldots,v_{2h})$ spans a labelled copy of $P_{2h}^{k,r}$ and $(v_1,\ldots, v_{h},v,v_{h+1},\ldots,v_{2h})$ spans a labelled copy of $P_{2h+1}^{k,r}$ in $H$.}
\end{definition}

\begin{proof}[Proof of Lemma~\ref{lem:abs}]
Let $k\ge 3$, $r\ge 2$ and $\alpha>0$ be given.
We define additional constants such that $1/n\ll 1/C,\varepsilon,\gamma\ll\beta\ll\alpha,1/k,1/r$, where $\beta=\beta(k,r,\alpha) $ is given by Lemma~\ref{l4}.
We split the proof into two parts.
We first find a set $\mathcal{F}$ consisting of disjoint absorbers in $V\setminus X$ and then connect them to a $(k,r)$-path by using Lemma~\ref{lem:con}.
We will expose $G=G^{(k)}(n,p)$ in two rounds: $G=G_1\cup G_2$ with $G_1$ and $G_2$ as independent copies of $G^{(k)}(n,p')$, where $(1-p')^2=1-p$.
Note that $p'>p/2\geq n^{-\sigma-\varepsilon}/2\geq n^{-\sigma-2\varepsilon}$.

Fix a vertex $v\in X$.
The first step is to find the candidates for $v$-absorbers $(v_1,\ldots,v_{2h})$ in $V\setminus X$ such that each of them induces a $k$-graph with edge set $\{vv_j\ldots v_{j+(k-2)}\colon\,j\in[k+2r-2]\}$.
Recall that the link $L_v$ of $v$ is the $(k-1)$-graph with vertex set $V(L_v)=V(H)\setminus\{v\}$ and edge set $E(L_v)=\left\{S\colon\,S\cup \{v\}\in E(H)\right\}$.
Our aim is to find many ordered $2h$-subsets $S=(v_1,\ldots,v_{2h})$ in $L_v[V\setminus X]$ such that $S$ induces a $(k-1)$-graph with edge set $\{v_j\ldots v_{j+(k-2)}\colon\,j\in[k+2r-2]\}$, i.e.~a tight path on $(v_1,\ldots,v_{2h})$ in $L_v[V\setminus X]$.
Since $\deg_H(S,V\setminus X)\geq \alpha n$ for every $S\in\binom{V}{k-1}$, we have $\delta_{k-2}(L_v[V\setminus X])\geq \alpha n$.
Then we get that
\[
e(L_v[V\setminus X])\geq \frac{\alpha n}{k-1}\binom{|V\setminus X|}{k-2}\geq\frac{\alpha n}{k-1}\binom{\alpha n}{k-2}\geq \frac{\alpha^{k-1}}{2(k-1)!}n^{k-1}.
\]

Applying Lemma~\ref{l4} with $H=L_v$, $s=2h$, $\theta=\frac{\alpha^{k-1}}{2(k-1)!}$, and $k-1$ in place of $k$, we conclude that $L_v[V\setminus X]$ contains at least $\beta n^{2h}$ such ordered $2h$-subsets $S$.
For convenience, let $X=\{x_1,\ldots,x_t\}$ and denote by $\mathcal{F}_i$ the collection of such ordered $2h$-sets for every $x_i\in X$.
Then we have $t\leq \gamma n$ and $|\mathcal{F}_i|\geq \beta n^{2h}$ for every $i\in [t]$.

The second step is to find an $x_i$-absorber for every $x_i\in X$ by adding the edges of $G_1$ to fill the missing edges while keeping them to be pairwise vertex-disjoint by Proposition~\ref{p1}.
Let $A$ and $W$ be the labelled $k$-graph and vertex set defined in Definition~\ref{def1}.
Note that the labelled $k$-graph $A$ is a copy of the $k$-graph induced by the desired missing edges.
We have $\Phi:=\Phi_{A, W}(n,p')\geq Cn$ and $\Phi':=\Phi_{A\setminus W}(n,p')\geq Cn$ by Lemma~\ref{l1} with $2\varepsilon$ in place of $\varepsilon$.
Applying Proposition~\ref{p1} to $F=A$, $W$ and the sets $\mathcal{F}_i$ for each $i\in[t]$, where $b=2h$, $w=1$ and $U_i=\{u_i\}$, we conclude that a.a.s.~there is a set of embeddings $\phi_1, \ldots , \phi_t$ such that each $\phi_i$ embeds a copy of $A$ into $G_1$ on $V$ with $W$ being mapped to $U_i$ and $V(A)\setminus W$ being mapped to a set in $\mathcal{F}_i$ which is vertex-disjoint with $\bigcup_{i\in[t]}U_i$.
Furthermore, for $i\neq j$ we have that $\phi_i(V (A)\setminus W) \cap \phi_j(V (A) \setminus W) = \emptyset$.
We use $\mathcal{F}=\left\{(x_1^i,\ldots,x_{2h}^i)\colon\,i\in[t]\right\}$ to denote the set of such absorbers for vertices of $X$.

Next we will connect these absorbers by using the edges of $H\cup G_2$.
For $i\in [t]$, we write $A_i=(x_1^i,\ldots,x_h^i)$ and $B_i=(x_{h+1}^i,\ldots,x_{2h}^i)$.
Let $V'=\bigcup_{i\in[t]}(A_i\cup B_i)$ and $R=V\setminus (X\cup V')$.
Then, for every $S\in \binom{V}{k-1}$, we have
$$\deg_H(S,R)\geq \deg_H(S,V\setminus X)-|V'|\ge \alpha n-2h\gamma n\geq \alpha n/2$$
by $\gamma\ll\alpha$.
Note that for each $i\in[t]$, the ordered $(2h)$-set $(x_1^i,\ldots,x_{2h}^i)$ spans a $(k,r)$-path with the ends $A_i$ and $B_i$.
It remains to connect these short $(k,r)$-paths to a single $(k,r)$-path $P_{abs}$.
This can be done by applying Lemma~\ref{lem:con} with $(V',R,\alpha/2,\gamma,2\varepsilon)$ playing the role of $(V',R,\alpha,\gamma,\varepsilon)$.
Consequently, $H\cup G_2$ contains a family of $(k,r)$-paths, each with exactly $b$ internal vertices in $R$, $\mathcal{T}=\{\mathcal{T}_1,\ldots, \mathcal{T}_t\}$ connecting $(B_i, A_{i+1})_{i\in [t-1]}$, which yields a $(k,r)$-path $P_{abs}$ and contains vertices from $V'\cup R$ only.
Furthermore, it is easy to see that $|V(P_{abs})|\leq 2h\gamma n+b\gamma n\leq \sqrt{\gamma} n$.
Since there is an $x_i$-absorber for each $x_i$ in $X$, clearly the
$(k,r)$-path $P_{abs}$ has the desired property.
\end{proof}
\section{Acknowledgements}

We thank the anonymous referees for their careful reading and helpful comments.

\end{document}